\documentclass[english]{amsart}
\usepackage{amsfonts,amsmath,amssymb,color,txfonts,amsthm,graphics,mathrsfs}

\usepackage[T1]{fontenc}
\usepackage[english]{babel}
\usepackage{color}
\usepackage[T1]{fontenc}
\usepackage[all]{xy}

\usepackage{enumerate}
\usepackage{amscd}
\usepackage{exscale}
\usepackage{latexsym}
\usepackage[all]{xy}
\usepackage{hyperref}
\usepackage{fancyhdr}
\usepackage{layout}
\usepackage{graphicx}

\newtheorem{theorem}{Theorem}
\newtheorem{lemma}{Lemma}[section]
\newtheorem{corollary}[lemma]{Corollary}
\newtheorem{proposition}[lemma]{Proposition}

\theoremstyle{definition}

\newtheorem{definition}[theorem]{Definition}

\newcommand\N{{\mathbb N}}

\newcommand\Q{{\mathbb Q}}

\newcommand\Res{{\mathrm{Res}}}

\newcommand\Z{{\mathbb Z}}

\newcommand\F{{\mathbb F}}

\newcommand\p{{\mathbb P}}

\newcommand\tr{\hbox to 1mm  {${}^t \!  $} }

\newcommand{\nc}{\newcommand}
\nc{\ace}{\`e }
\nc{\aca}{\`a }
\nc{\aci}{\`i }
\nc{\aco}{\`o }
\nc{\acu}{\`u }
\nc{\pid}{\mathfrak{p} }
\nc{\bdm}{\begin{displaymath}}
\nc{\edm}{\end{displaymath}}
\nc{\beq}{\begin{equation}}
\nc{\eeq}{\end{equation}}
\nc{\dpid}{\delta_{\mathfrak{p}}}
\nc{\vvs}{\textquotedblleft}
\nc{\vvd}{\textquotedblright}
\nc{\os}{\mathcal{O_S}}

\nc{\mcu}{\mathcal{U}}
\nc{\srsu}{\sqrt{R_S^*}}
\nc{\srs}{\sqrt{R_S}}
\nc{\colb}{\color{blue}}
\nc{\noc}{\normalcolor}

\title[Preperiodic points over global fields]{Preperiodic points for rational functions defined over a global field in terms of good reduction}
\subjclass[2010]{37P05, 37P35, (primary), 11D45 (secondary)}

\author{Jung Kyu Canci}
\address{Jung Kyu Canci, Universit\"{a}t Basel, Mathematisches Institut, Rheinsprung $21$, CH-$4051$ Basel}
\email{jungkyu.canci@unibas.ch}

\author{Laura Paladino}
\thanks{L. Paladino is partially supported by Istituto Nazionale di Alta Matematica, grant research \emph{Assegno di ricerca Ing. G. Schirillo}, and partially supported by the European Commission and by Calabria Region through the
European Social Fund.}
\address{Laura Paladino, Universit\aca di Pisa, Dipartimento di Matematica, Largo Bruno Pontecorvo 5, 56127 Pisa, Italy.}
\email{paladino@mail.dm.unipi.it}

\begin{document}

\begin{abstract}
Let $\phi$ be an endomorphism of the projective line defined over a global field $K$. We prove a bound for the cardinality of the set of $K$--rational preperiodic points for $\phi$ in terms of the number of places of bad reduction. The result is completely new in the function field case and it is an improvement of the number field case.\end{abstract}

\maketitle

\section{introduction}
Let $\phi\colon \p_1\to \p_1$ be a rational function defined over a field $K$. A point $P$ is said to be \emph{periodic} for $\phi$ if there exists an integer $n>0$ such that $\phi^n(P)=P$. We call \emph{minimal period} the minimal number $n$ with the above property.
We say that $P$ is a \emph{preperiodic point} for $\phi$ if its (forward) orbit $O_\phi(P)=\{\phi^n(P)\mid n\in\N\}$ contains a periodic point, that is equivalent to say that the orbit $O_\phi(P)$ is finite.  The orbit of a periodic point is called a \emph{cycle} and its size is called the \emph{length of the cycle}.

Let $K$ be a global field, i.e. $K$ is either a finite extension of the field $\mathbb{Q}$ or a finite extension of the field $\F_p(t)$, where $p$ is a prime number and $\F_p$ is the field with $p$ elements. Let ${\rm PrePer}(\phi,K)$ be the set of $K$--rational preperiodic points for $\phi$. By considering the notion of height, one can verify that the set ${\rm PrePer}(\phi,K)$ is finite for any rational map $\phi\colon\p_1\to\p_1$ defined over $K$ (see for example \cite{Z} or \cite{HS}). The finiteness of the set ${\rm PrePer}(f,K)$ follows by applying \cite[Theorem B.2.5, p.179]{HS} and \cite[Theorem B.2.3, p.177]{HS} (these last theorems are stated in the case of number fields, but with similar proofs one verifies the analogous statements in the function field case). Anyway, from the above two theorems  one can deduce a bound that depends strictly on the coefficients of the map $\phi$ (see also \cite[Exercise 3.26 p.99]{Z}).  In this context there is the so-called Uniform Boundedness Conjecture formulated in \cite{MS1} by Morton and Silverman. It says that for any number field $K$, the cardinality of the set ${\rm PrePer}(\phi,K)$ of a morphism $\phi\colon \p_N\to\p_N$ of degree $d \geq 2$, defined over $K$, is bounded by a number depending only on the integers $d,N$ and on the degree $D$ of the extension $K/\Q$. It seems very hard to solve this conjecture.  An example to give an evidence of the difficulties is provided by the polynomial case, where it is conjectured that a polynomial of degree $2$, defined over $\mathbb{Q}$, admits no rational periodic points of order $n>3$, see \cite[Conjecture 2]{FPS}. This last conjecture is proved only for $n=4$ \cite[Theorem 4]{Mor} and $n=5$ \cite[Theorem 1]{FPS}. Some evidence for $n=6$ is given in \cite[Section 10]{FPS}, \cite{Sto} and \cite{HuIn}.
Furthermore, by considering the Latt\`es map associated to the multiplication by two map $[2]$ over an elliptic curve $E$, it is possible to see that the Uniform Boundedness Conjeture for $N=1$ and $d=4$ implies  Merel's Theorem on torsion points of elliptic curves (see \cite{Me}).    The Latt\`{e}s map has degree 4 and its preperiodic points are in one-to-one correspondence with the torsion points of $E/\{\pm 1\}$ (see \cite{Sil.2}). The aim of our work is to prove a weaker form of the Uniform Boundedness Conjecture, over  all   global field, where the constant depends on one more parameter, that is the number of  primes of bad reduction.

The notion of good (and bad) reduction considered in the present article is the following one:  let $K$ be a global field, $R$ its ring of algebraic integers, $\pid$ a non zero prime ideal of $R$ and $R_\pid$ the local ring at $\pid$; we say that an endomorphism $\phi$ of $\p_1$ has good
reduction at $\pid$ if $\phi$ can be written in the form $\phi([x:y]) = [F(x,y),G(x,y)]$, where $F(x,y)$ and $G(x,y)$ are homogeneous polynomial of the same degree, with coefficients in $R_\pid$ and such that their resultant
$\Res(F,G)$  is a $\pid$--unit. This notion of good reduction was introduced by Morton and Silverman in \cite{MS}.

The first author already studied some problems linked to the Uniform Boundedness Conjecture. In particular, he studied the case when $N=1$ in the number field case and he took in consideration families of rational functions characterized in terms of good reduction too. In \cite[Theorem 1]{C} he proved the following fact: let $K$ be a number field and $S$ be a finite set of places of $K$ containing all the archimedean ones. Let $\phi\colon\p_1\to\p_1$ be an endomorphism defined over $K$ with good reduction outside $S$ (i.e. good reduction at each $\pid\notin S$). Then the orbit of a preperiodic point $P\in\p_1(K)$ has cardinality bounded by a number $c(|S|)$ which depends only on the number  $|S|$ of elements in $S$.
The main aim of our work was to prove a similar result in the function field case. But the techniques that we found work also in the number field case and in that case we obtain a better bound than the one proved in \cite{C}. We resume  those  results in the following theorem.

\begin{theorem}\label{preper} Let $K$ be a global field. Let $S$ be a finite set of places of $K$, containing all the archimedean ones, with cardinality $|S|\geq 1$. Let $p$ be the characteristic of $K$. Let $D=[K:\mathbb{\F}_p(t)]$ when $p>0$, or $D=[K:\mathbb{Q}]$ when $p=0$.  Then there exists a number $\eta(p,D,|S|)$, depending only on $p$, $D$ and  $|S|$, such that if  $P\in\p_1(K)$ is a preperiodic point for an endomorphism   $\phi$ of $\p_1$ defined over $K$   with good reduction outside $S$,   then $|O_\phi(P)|\leq  \eta(p,D,|S|)$.
We can choose
$$\eta(0,D,|S|)=\max\left\{(2^{16|S|-8}+3)\left[12|S|\log(5|S|)\right]^{D}, \left[12(|S|+2)\log(5|S|+5)\right]^{4D}\right\}$$
 in zero characteristic  and
\begin{equation}\label{eta}\eta(p,D,|S|)=(p|S|)^{4D}\max\left\{\left(p|S|\right)^{2D},p^{4|S|-2}\right\}.\end{equation}
 in positive characteristic.
\end{theorem}

  Note that the bound does not depend on the degree of the endomorphism $\phi$.   The condition $|S|\geq 1$ is only a technical one. In the case of number fields, we require that $S$ contains all archimedean places, then it is clear that the cardinality of $S$ is not zero.   In the case of function fields the situation is quite different  from  the case of number  fields.   For example, all places are non archimedean.  The condition that $S$ is not empty is important in order to have that the ring of $S$--integers $R_S$ is different from its field of constants. Also the fact that the class number of $R_S$ is finite will play an important role (see Lemma 5.6 ind Proposition 14.1 in \cite{Ros} Proposition 14.1).


The result stated in Theorem \ref{preper}  extends to all  global fields and to preperiodic points the result proved by Morton and Silverman in \cite[Corollary B]{MS1}. They proved the bound $12(r+2)\log(5(r+2))^{4[K:\Q]}$ for the length of a cycle of a $K$--rational periodic point for an endomorphism $\phi\colon\p_1\to\p_1$, defined over a number field $K$, with  at most  $r$ primes of bad reduction.
Their bound reposes on the result proved in  \cite[Proposition 3.2(b)]{MS}. To produce that bound they considered the reduction modulo two suitable primes in $K$, i.e. they considered the reduction to two reduced fields having two different characteristics. Their  technique does not work in the function field case. Our proof uses an  $S$-unit equation Theorem in positive characteristic. More precisely, we use a theorem in two $S$-units (see
Theorem \ref{v}), that is essentially  \cite[Theorem 1]{Vol} where we consider   also the case of inseparable extensions. With function fields, a difficulty is that there could be infinitely many solutions in $S$--units even for an equation in two variables. For example, if we take $K=\F_p(t)$ and $S=\langle t,1-t\rangle$, then the equation $x+y=1$ admits the solutions $(x,y)=(t^{p^n},(1-t)^{p^n})$ for  each integer $n>0$  (see \cite{Le1} and \cite{Le2} for a complete description of the  solutions  of $x+y=1$  with the above $S$).  For some results in a more general setting see \cite{DM}. We shall use some ideas already contained in \cite{C.1} and \cite{C}, but the original idea of using $S$--unit theorems in Arithemitc of Dynamical System is due to Narkiewicz \cite{N.1}.
As an application of our Theorem \ref{preper} we have the following result.

\begin{corollary}\label{UBCforCGR} Let $K$ be a global field. Let $S$ be a finite set of places of $K$ of cardinality $|S|\geq 1$, containing all the archimedean ones. Let $p$ be the characteristic of $K$.  Let $D=[K:\mathbb{\F}_p(t)]$ when $p>0$, or $D=[K:\mathbb{Q}]$ when $p=0$.  For any integer $d\geq 2$, let $\textrm{Rat}_{d,S}(K)$ be the set of the endomorphisms of $\p_1$ of degree $d$, defined over $K$ and with good reduction outside $S$.
Then there exists a number $C=C(p,D,d,|S|)$, depending only on $p$, $D$, $d$ and  $|S|$, such that for any endomorphism $\phi\in \textrm{Rat}_{d,S}(K)$,   we have $\#{\rm PrePer}(\phi, \p_1(K))\leq C(p,D,d,|S|)$.\end{corollary}

 \noindent Our Corollary \ref{UBCforCGR} is a sort of generalization of the result proved by Benedetto in \cite{B.1}. He studied dynamics given by the maps induced by polynomials $\phi(z)\in K[z]$. Benedetto's bound is quite sharp, it is of the form $O(|S|\log|S|)$ where the constant in the big $O$ depends only on the degree $d$ of the polynomial $\phi$ and the degree $D$ of the extension.  His proof involves the study of the filled Julia set associated to a polynomial $\phi$. We use a completely different approach. Our techniques of proof could give only a very big estimation for the number $C(p,D,d,|S|)$ (for this reason we decided not to give an explicit estimation for $C(p,D,d,|S|)$), but our result holds for any rational map in $K(z)$.

The techniques that we use to prove Theorem \ref{preper} can be used to prove small bounds
in some particular situations, as in the case of the next corollary.

\begin{corollary} \label{n3}
Let  $\phi:\p_1\rightarrow \p_1$ be an endomorphism defined over $\mathbb{Q}$, with good reduction at every non-archimedean place.
\begin{itemize}
\item If  $P\in \p_1({\mathbb{Q}})$ is a periodic point for $\phi$
with minimal period $n$, then $n\leq 3$.
\item  If $P\in \p_1({\mathbb{Q}})$ is a preperiodic point for $\phi$, then $|O_\phi(P)|\leq 12$.
\end{itemize}
\end{corollary}
%
%

Effective bounds as in Theorem \ref{preper} can be also useful to solve problems concerning torsion points of elliptic curves.
  For instance, in some previous papers, the second author was faced with the local-global divisibility problem on elliptic curves (for example see \cite{Pal3}, \cite{PRV} and see also \cite{DZ}).
If $\mathcal{E}$ is an elliptic curves defined on a number field $K$ that  does not contain ${\mathbb{Q}}(\zeta_p+\zeta_p^{-1})$ (where $\zeta_p$ is a $p$-root of unity) and there exist no $K$-rational torsion points with exact order a prime $p$, then the local-global divisibility by $p^n$ holds for every positive integer $n$ \cite{PRV}. Therefore Theorem \ref{preper} gives a bound $C(D,|S|):=\eta(0,D,4,|S|)$ to the number of primes $p$ for which the local-global divisibility may fail. One knows already some bounds that depend only on the degree $D$ of the extension (e.g. the ones provided by Merel \cite[Proposition 2]{Me}, by Oesterl\'{e} \cite{Oe} and by
Parent \cite[Corollary 1.8]{Par}). Our result  provides just another point of view to the above problem and in some particular cases could provide some small bounds.

It could be interesting to study the same problem about preperiodic points of a rational map of $\p^1(K)$ in the situation when $K$ is a function field in zero characteristic. In this case one could apply the Evertse and Zannier's result contained in \cite{EZ}.

Here there is a short overview of the contents of the paper. In section \ref{prel} we present the tools that we shall use in our proofs. In section \ref{blc}, we prove a bound for the minimal periodicity of periodic points in the case of function fields. Section \ref{bcfo} contains the proof of Theorem \ref{preper}, Corollary \ref{UBCforCGR} and Corollary \ref{n3}.

\emph{Acknowledgements.} We thank Dominik Leitner and David Masser for useful discussions. We thank Sebastian Troncoso for pointing out an inaccuracy in Lemma 3.2.
The article was written when the second author was at the University of Basel; in particular she thanks the Department of Mathematics. We would like also to thank an anonymous referee that suggested to use the content of Lemma \ref{=p} that gave a significant improvement of the bounds in our Theorem \ref{preper}.

\section{Preliminaries}\label{prel}
Throughout the whole paper
   we shall use the following notation:
let $K$ be a global field  and $\bar{K}$ its algebraic closure; let $v_\pid$ be the normalized valuation on $K$ associated to a non archimedean place $\pid$ such that $v_\pid(K)=\mathbb{Z}$.
 Let $R_\pid$ be the local ring $\{x\in K\mid v_\pid (x)\geq 1\}.$  As usual, we still denote by $\pid$ the maximal ideal in $R_\pid$.
 Let $k(\pid)$ be the residue field and $p$ its characteristic. Since $R_\pid$ is a principal ideal domain, then   there exists a canonical reduction map $\p_1(K)\to \p_1(k(\pid))$, that maps a point $P$ to a point $\tilde{P}\in\p_1(k(\pid))$  called the reduction of $P$ modulo $\pid$.

When $K=\F_p(t)$ all places are exactly the ones associated either to a monic irreducible polynomial in $\F_p[t]$ or to the place at infinity given by the valuation $v_\infty(f(x)/g(x))=\deg(g(x))-\deg(f(x))$, that is the valuation associated to $1/x$. All these places are non-archimean, i.e. $v_\pid(x+y)\geq \min \{v_\pid(x),v_\pid(y)\}$ for each $x,y\in K$. In an arbitrary finite extension $K$ of $\F_p(t)$, each valuation of  $K$ extends one of $\F_p(t)$.
We have a similar situation in the number field case. The non archimedean places in $\Q$ are the ones associated to the valuations at any prime $p$ of $\Z$. But there is also a place that is not non--archimedean. It is the one associated to the usual absolute value on $\Q$. With an arbitrary number field $K$ the archimedean places are the ones that extend  the usual absolute value on $\Q$.

For every finite set $S$ of places of $K$, containing all the archimedean ones, we shall denote by
$R_S \coloneqq\{x\in K \mid v_{\mathfrak{p}}(x)\geq0 \ \text{for every prime }\ \mathfrak{p}\notin S\}$
the ring of $S$-integers and by
$ R_S^\ast \coloneqq\{x\in K^\ast\mid v_{\mathfrak{p}}(x)=0 \ \text{for every prime }\ \mathfrak{p}\notin S\}$
the group of $S$-units.

\subsection{Reduction of cycles}
We shall use the notion of good reduction already given in the introduction. In other words we say that a morphism $\phi:\p_1\to\p_1$ has good reduction at $\pid$ if there exist $F,G\in R_p[X,Y]$ homogeneous polynomials of the same degree, such that $\phi[X:Y]=[F(X,Y):G(X,Y)]$ and the reduced map   $\phi_\pid$, obtained by reducing the coefficients of $F$ and $G$ modulo $\pid$, has the same degree of $\phi$.
Otherwise we say that it has bad reduction. Given a set $S$ of places of  $K$ containing all the archimedean ones, we say that $\phi$ has good reduction outside $S$ if it has good reduction at any place $\pid\notin S$.

 If an endomorphism of $\p_1$ has good reduction, then we have some important information on the length of a cycle. In this direction an important tool in our proof is the next result, proved by Morton and Silverman in
and \cite{MS1}, or independently by Zieve in his PhD thesis \cite{Zie} (here we state a version adapted to our setting).

\begin{theorem}[Morton and Silverman \cite{MS1}, Zieve \cite{Zie}]\label{mst} Let $K,\pid,p$ be as above. Let $\phi$ be an endomorphism of $\p_1$ of degree at least two defined over $K$ with good reduction at $\pid$.
Let $P\in \p_1(K)$ be a periodic point for $\phi$ with minimal period $n$.  Let $\widetilde{P}$ be the reduction of $P$ modulo $\pid$, $m$ the minimal  period of $\widetilde{P}$ for the map $\phi_p$  and $r$ the multiplicative period of $(\phi^m)'(P) $ in $k(\pid)\setminus\{0\}$. Then one of the following three conditions holds

\begin{itemize}
  \item[(i)] $n=m$;
  \item[(ii)] $n=mr$;
  \item[(iii)] $n=p^emr$, for some $e\geq 1$.
\end{itemize}
\end{theorem}

  In the notation of Theorem \ref{mst}, if $(\phi^m)'(P)=0$ modulo $\pid$, then we set $r=\infty$. If $P$ is a periodic point, then  (ii) and (iii) are not possible with $r=\infty$.
The above theorem will be useful to bound the length of a cycle in terms of primes of bad reduction. In particular, it will be useful to apply some divisibility arguments contained in the following subsection.

\subsection{Divisibility arguments}\label{dia}
First of all we fix some notation.

Let $P_1=\left[x_1:y_1\right],P_2=\left[x_2:y_2\right]$ be two distinct points in $\mathbb{P}_1(K)$. By using the notation of  \cite{MS} we shall denote by
\begin{center}$\delta_{\mathfrak{p}}\,(P_1,P_2)=v_{\mathfrak{p}}\,(x_1y_2-x_2y_1)-\min\{v_{\mathfrak{p}}(x_1),v_{\mathfrak{p}}(y_1)\}-\min\{v_{\mathfrak{p}}(x_2),v_{\mathfrak{p}}(y_2)\}$\end{center}
the $\mathfrak{p}$-adic logarithmic distance; $\delta_{\mathfrak{p}}\,(P_1,P_2)$ is independent of the choice of the homogeneous coordinates, i.e. it is well defined. The logarithmic distance is always non negative and $\delta_{\mathfrak{p}}(P_1,P_2)>0$ if and only if $P_1$ and $P_2$ have the same reduction modulo $\pid$.

The divisibility arguments, that we shall use to produce the $S$--unit equations useful to prove our bounds, are obtained starting from the following two facts:

\begin{proposition}\label{5.1}\emph{\cite[Proposition 5.1]{MS}} For all $P_1,P_2,P_3\in\mathbb{P}_1(K)$, we have
\par\medskip \centerline{$\delta_{\mathfrak{p}}(P_1,P_3)\geq \min\{\delta_{\mathfrak{p}}(P_1,P_2),\delta_{\mathfrak{p}}(P_2,P_3)\}$.}\end{proposition}

\begin{proposition}\label{5.2}\emph{\cite[Proposition 5.2]{MS}}
Let $\phi\colon \p_1\to\p_1$ be a morphism defined over $K$ with good reduction at $\pid$. Then for any $P,Q\in\p(K)$ we have
$\delta_{\mathfrak{p}}(\phi(P),\phi(Q))\geq \delta_{\mathfrak{p}}(P,Q)$.
\end{proposition}

As a direct application of the previous propositions we have the next proposition.

\begin{proposition}\label{6.1}\emph{\cite[Proposition 6.1]{MS}}
Let $\phi\colon \p_1\to\p_1$ be a morphism defined over $K$ with good reduction at $\pid$. Let $P\in\p(K)$ be a periodic point for $\phi$ with
minimal period n. Then \par\medskip
\begin{itemize}
\item $\dpid(\phi^i(P),\phi^j(P))=\dpid(\phi^{i+k}(P),\phi^{j+k}(P))$\ \ for every $i,j,k\in\N$.
\item Let $i,j\in\N$ be  such that $\gcd(i-j,n)=1$. Then $\dpid(\phi^i(P),\phi^j(P))=\dpid(\phi(P),P)$.
\end{itemize}
\end{proposition}

\subsection{$\srs$--coprime coordinates}\label{scc}
Let $P\in \p_1(K)$ where $K$ is an arbitrary global field. Let $S$ be a finite non empty set of places containing all the archimedean ones.  There exist $a,b\in R_S$ such that $P=[a:b]$. We say that $[a:b]$ are $S$--coprime coordinates for $P$ if $\min\{v_\pid(a), v_\pid(b)\}=0$ for each $\pid\notin S$.  If the ring $R_S$ is not a principal ideal domain, there exist points in $\p_1(K)$ that does not have $S$--coprime coordinates. We could avoid  that problem  by taking  an enlarged set $\mathbb{S}$ of places of $K$ containing $S$, such that the ring $R_{\mathbb{S}}$ is a principal ideal domain.  Indeed the class number of $R_S$ (i.e. the order of the fractional ideal class group) is finite (see Corollary in Chapter 5 in \cite{M.1} for number fields and  both Lemma 5.6 and  Proposition 14.1 in \cite{Ros} for function fields). We denote by $h$ the class number of $R_S$. By a simple inductive argument, we can choose $\mathbb{S}$ such that $|\mathbb{S}|\leq s+h-1$. But working with $\mathbb{S}$, we will obtain a bound in Theorem \ref{preper} depending also on $h$. We use the same argument as in \cite{C} to avoid the presence on $h$ in our bounds. For the reader convenience we write below this argument.

Let ${\bf a_1},\ldots,{\bf a_h}$ be ideals of $R_S$ that form a full system of representatives for the ideal classes of $R_S$. For each $i\in\{1,\ldots,h\}$ there is an $S$-integer $\alpha_i\in R_S$ such that ${\bf a}_i^h=\alpha_iR_S.$
 Let $L=K(\zeta,\sqrt[h]{\alpha_1},\ldots, \sqrt[h]{\alpha_h})$, where $\zeta$ is a primitive $h$--th root of unity and $\sqrt[h]{\alpha_1}$ is a chosen $h$--th root of $\alpha_i$.

 Let $\hat{S}$ be the set of places of $L$ lying above the places in $S$. Let $R_{\hat{S}}$ and $R_{\hat{S}}^*$ be respectively  the ring of $\hat{S}$--integers and the group of $\hat{S}$--units  in $L$. We denote by $\sqrt{R_S^\ast}$, $\sqrt{R_S}$ and $\sqrt{K}$ the radical in $L^\ast$ of the groups $R_S^\ast$, the ring $R_S$ and $K$ respectively. It turns out that  $\sqrt{R_{S}^\ast}=R_{\hat{S}}^*\cap \sqrt{K}^*$ and $\sqrt{R_{S}}=R_{\hat{S}}\cap \sqrt{K}$; furthermore $\sqrt{R_{S}^\ast}$ is a subgroup of $L^\ast$ of rank $|S|-1$ in the case of number field (see \cite{C}). In the case of function field we have that the group $R_S^\ast/\left(K^*\cap \overline{\F}_p\right)$  has finite rank equal to $|S|-1$ (e.g. see \cite[Proposition 14.2 p. 243]{Ros}). Thus, since $K\cap \overline{\F}_p$ is a finite field, we have that $R_S^*$ has rank $|S|$ and then  also the group $\sqrt{R_{S}^\ast}$ has rank $|S|$.
   For each $P\in \p_1(K)$, there exist two $x,y\in R_S$ such that $P=[x:y]$. Let ${\bf a_i}$ be the representatives in the same ideal class of $xR_s+yR_S$. Let $\alpha_i\in R_S$ be such that ${\bf a_i}^h=\alpha_i R_S$.  Hence there exists $\lambda_i\in K$ such that $(xR_S+yR_S)^h=\lambda_i^h\alpha_i R_S$. Let $x^\prime=x/(\lambda_i\sqrt[h]\alpha_i)$ and $y^\prime=y/(\lambda_i\sqrt[h]\alpha_i)$.  Then  $x^\prime R_{\hat{S}}+y^\prime R_{\hat{S}}=R_{\hat{S}}$ and so $x^\prime,y^\prime\in \sqrt{K^*}\cap R_{\hat{S}}=\sqrt{R_S}$. More precisely, it is possible to see that there exist two elements $a,b\in \srs$ such that $x^\prime a+y^\prime b=1$. Of course $P=[x^\prime:y^\prime]$. In this case we say that $ P$ is written in $\sqrt{R_{S}}$--coprime integral coordinates.
   We shall use the divisibility arguments contained in Subsection \ref{dia} for points in $\p_1(K)$  written  in $\srs$--coprime coordinates, instead of $S$-coprime coordinates.  Therefore the results contained in Subsection \ref{dia} will be applied with $L$ instead of $K$ and $\hat{S}$ instead of $S$. We will always assume that the points in $\p_1(K)$ are written in $\sqrt{R_{S}}$--coprime integral coordinates.

\subsection{On the equation $ax+by=1$ in function fields}
Let $K$ be a global function field.
Let $S$ be a finite fixed set of places of $K$. Let $L$, $\sqrt{R_S}$, $\sqrt{R_S^*}$ be defined as in the previous Subsection \ref{scc}.
We use the classical notation $\overline{\F}_p$ for the algebraic closure of $\F_p$. The case when $S=\emptyset$ is trivial, because the ring of $S$--integers is already finite; more precisely $R_S=R_S^*=K^*\cap \overline{\F}_p$. Then in what follows we assume $S\neq \emptyset$.


\begin{definition}
An equation $ax+by=1$, with $a,b\in L^*$, is called
\emph{S-trivial} if there exists an integer $n$, coprime with $p$, such that $a^n,b^n\in R_S^*$  (see \cite{Vol}).
\end{definition}

Recall that if $L$ is a separable extension of $\F_p(t)$, then  the standard derivation of $\F_p(t)$ extends uniquely to $K$ (see e.g. \cite{Sti}).
If $L$ is not a separable extension of $\F_p(t)$, we could have some technical problems; for example it is not clear how to extend the standard derivation of $\F_p(t)$ on $L$. Anyway, a field extension $L/\F_p(t)$ splits in the composition of two extensions $L/L_s$ and $L_s/\F_p(t)$, where
$L_s/\F_p(t)$ is separable and $L/L_s$ is purely inseparable (see for example \cite[\S 3.10 and App. A]{Sti} or see \cite{Le1} and \cite{Le2} for a summary of these arguments).
 This composition of extensions will be useful  in the proof of the following  statement.

\begin{theorem}  \label{v}
Let $K$ be a finite extension of  the field $\F_p(t)$.
Let $S$ be a finite set of
places of $K$ with cardinality $|S|\geq 1$. Let $L$ and $\sqrt{R_S^*}$ as defined in Subsection \ref{scc}. For any fixed $a,b\in L^*$, if the equation
\beq\label{equa}ax+by=1\eeq
is not $S$-trivial, then it has at most $r(p,|S|)=p^{2|S|-2}(p^{2|S|-2}+p-2)/(p-1)$ solutions $(x,y)\in (\sqrt{R_S^*})^2$.
\end{theorem}

\begin{proof}
\emph{Case $L$ separable over $\F_p(t)$}.  In this case Theorem \ref{v} is just \cite[Theorem 1]{Vol} adapted to our situation. Note that $\srsu=R_S^*$ and $|(R_S^*)^2/H|=p^{2|S|-2}$, where $H=\{(x,y)\in (R_S^*)^2\mid Dx=Dy=0\}$.

\emph{Case $L$ inseparable over $\F_p(t)$}. Let $L_s$ be the subfield of $L$ such that $L/\F_p(t)$ splits in the composition of two extensions $L/L_s$ and $L_s/\F_p(t)$ where
$L_s/\F_p(t)$ is separable and $L/L_s$ is purely inseparable. Recall that every prime  in  $L_s$ extends to a unique prime  in  $L$ (see \cite{Sti}). Thus the group  $\widetilde{\sqrt{R_S^*}}:=\sqrt{R_S^*}\cap K_s$ has rank $|S|$. Let $k$ be the integer such that $[K:K_s]=p^k$. The existence of such a number $k$ follows from the structure of the purely inseparable extensions; e. g. see \cite[Corollary 6.8 p.250]{L.2}. If we take the $p^k$--power of both sides in (\ref{equa}), we get
$1=(ax+by)^{p^k}=a^{p^k}x^{p^k}+b^{p^k}y^{p^k}.$
Therefore if $(x,y)\in \left(\sqrt{R_S^*}\right)^2$ is a solution of (\ref{equa}), then $(X,Y)=(x^{p^k},y^{p^k})\in  \left(\widetilde{\sqrt{R_S^*}}\right)^2$ is a solution of $AX+BY=1$, where $A=a^{p^k}, B=b^{p^k}$ belong to $K_s$.  Hence the problem reduces to the study of the solutions of $AX+BY=1$, with $(X,Y)\in  \left(\widetilde{\sqrt{R_S^*}}\right)^2$, in the separable case. Indeed  any solution $(x,y)\in \left(\sqrt{R_S^*}\right)^2$ for the equation (\ref{equa}) corresponds to a solution  $(X,Y)=(x^{p^k},y^{p^k})\in \left(\widetilde{\sqrt{R_S^*}}\right)^2$ for  $AX+BY=1$. Note that the correspondence $(x,y)\to (x^{p^k},y^{p^k})$ is injective.
\end{proof}

\subsection{On the equation $ax+by=1$ in number fields}\label{senf}

Let $K$ be a number field and let $S$ be a finite fixed set of places of $K$, containing all the archimedean ones.

\begin{theorem}[\cite{B.S.1}]\label{bsb} Let $L$ be a number field and let $\Gamma$ be a subgroup of $(L^\ast)^2$ of rank $r$. Then the equation $x+y=1$
has at most $2^{8(r+1)}$ solutions with $(x,y)\in\Gamma$.
\end{theorem}
In the following we shall  use Theorem \ref{bsb} with $\Gamma=\left(\sqrt{R_{S}^*}\right)^2$.

\section{Bound for the length of a cycle}\label{blc}
 The aim of this section is to prove a result, in positive characteristic, similar to the following one, proved by Morton and Silverman for number fields. Here we present the statement of \cite[Corollary B] {MS1}  adapted to our notation.

\begin{theorem}\label{msb}\emph{(\cite[Corollary B]{MS1})} Let $K$ be a number field. Let $S$ be a finite set of places of $K$. Let $\phi$ be an endomorphisms of $\p_1$ of degree $d\geq 2$ defined over $K$ with good reduction outside $S$. Let $P\in\mathbb{P}_1(K)$ be a periodic point for $\phi$ with minimal period $n$, then
\begin{equation}\label{n} n\leq \left[12(|S|+1)\log(5(|S|+1))\right]^{4\left[K:\mathbb{Q}\right]}.\end{equation}
\end{theorem}

   Because of Theorem \ref{msb},  in what follows we assume that $K$ is a global function field with degree $D$ over $\F_p(t)$.  As usual $S$ is a finite non empty set of places of $K$. Let $L, \hat{S}, \sqrt{R_S}$ and $\sqrt{R_S^*}$ be  as defined  in Subsection \ref{scc}.

If we take two points $P_1=[x_1:y_1]$ and $P_2=[x_2:y_2]$ written in $\srs$--coprime integral coordinates, we have that
$\dpid (P_1,P_2)=v_\pid(x_1y_2-x_2y_1) $ for each $\pid\notin \hat{S}$. \par\medskip

To bound the length of the cycle of a preperiodic point of a rational map of $\p_1$ in terms of $p, D$ and $|S|$, we first need to
prove the existence of a prime ${\mathfrak{p}}\notin S$ such that $|k(\pid)|$ is bounded in terms of $p, D$ and $|S|$.

\begin{lemma}\label{ipDs} Let $K,p,D $ be as above. There exists  a number $i(p,D,|S|)$, that depends only on $p,D$ and $|S|$, and a prime $\pid\notin S$ such that the corresponding residue field $k(\pid)$ has cardinality bounded by $i(p,D,|S|)$. We can take  $i(p,D,|S|)=(p|S|)^{2D}-1$.
\end{lemma}
\begin{proof}
Suppose that $D=1$ (i.e. $K=\F_p(t)$). 
We claim that there is a prime $\pid\notin S$ such that
\beq\label{ps2}|k(\pid)|< (p|S|)^2.\eeq
Recall that the number of monic irreducible polynomials in $\F_p[t]$ of degree $n$  is given by
$I(n)= \frac{1}{n}\sum_{d|n}\mu(n/d)p^d$,
where $\mu$ denotes the M\"obius function (e.g. see \cite[Corollary at p.13]{Ros}). Let $N$ be such that
\beq\begin{array}{c}\label{Ns}\sum_{n\leq N}I(n)> |S|.\end{array}\eeq
Thus there is a finite prime $\pid\notin S$, whose associated monic irreducible polynomial is $\pi$, with $\deg\pi\leq N$. Hence $|k(\pid)|\leq p^N$.

Note that for $n\leq 3$ we easily see that $I(n)\geq \frac{p^n}{2n}$. We want to show that the inequality holds for any $n$. Indeed for $n\geq 4$
\beq\label{In}\begin{array}{c}I(n)\geq \frac{1}{n}\left(p^n-\sum_{d\leq n/2}p^d\right)\geq \frac{1}{n}\left(p^n-\frac{p^{n/2+1}-1}{p-1}\right)\geq \frac{1}{n}\left(p^n-2p^{n/2}\right)\geq \frac{p^n}{2n}.\end{array}\eeq

Suppose that $S$ and $p$ are such that $|S|>1$ and $\frac{p^{|S|-1}}{2(|S|-1)}>|S|$.  We are excluding the three cases when i) $|S|= 1$;  ii) $p=2$ and $|S|\leq 7$; iii) $p=3$ and $|S|\leq 3$. Let $N$ be the smallest integral number such that $\frac{p^N}{2N}>|S|>N$. Such a number $N$ exists because of our assumption on $|S|$ and $p$. By (\ref{In}) and $\frac{p^N}{2N}>|S|$, there exists $\pid\notin S$ of degree $N$ such that (\ref{Ns}) holds. Indeed, if $p^{N}\geq (p|S|)^2$ we would have $p^{N-1}>2(N-1)|S|$, that  contradicts  the minimality of $N$. If $|S|=1$, it is clear that there is a prime $\pid\notin S$ such that (\ref{ps2}) holds. Let $p=2$ and $|S|=2$. We  have that there exists a monic irreducible polynomial $\pid\notin S$ of degree 2. Thus we have $4=|k(\pid)|< 8=(p|S|)^2$.  When $p=2$ and $3\leq |S|\leq 7$,  take $N=4$. The sum in (\ref{Ns}) is 8, then there is a monic polynomial $\pid\notin S$ of degree $4$. So $|k(\pid)|=16< 36\leq (2|S|)^2$.  Similar arguments work when $p=3$ and $|S|\leq 3$.

For arbitrary finite extension of $\F_p(t)$  of degree $D\geq 1$,  it suffices to remark that $\pid\cap  \F_p(t)$ is generated by a monic irreducible polynomial $\pi$ in $\F_p[t]$, for each prime ideal $\pid$ of $R$.  The prime $\pid$ is said a prime above $\pi$  or equivalently that $\pi$ is below $\pid$. The cardinality of the set of primes $\pi$ that are below the primes in $S$ is bounded by $|S|$.   Thus, there exists a prime $\pid\notin S$ above a prime $\pi$ in $\F_p[t]$, such that $|k(\pid)|\leq |\F_p(t)(\pi)|^D$. By applying the inequality \eqref{ps2}, we have $|k(\pid)|\leq |\F_p(t)(\pi)|^D<(p|S|)^{2D}$. Hence we can take  $i(p,D,|S|)=(p|S|)^{2D}-1$.
\end{proof}

The following lemma is an elementary application of the previous result,  that  will be useful in the rest of the paper.

\begin{lemma}\label{=p} Let $K$ be a function field of degree $D$ over $\F_p(t)$ and $S$ a non empty finite set of places of $K$. Let $P_i\in \p_1(K)$ with $i\in\{0,\ldots n-1\}$ be $n$ distinct points such that

\beq\label{=}\dpid(P_0,P_1)=\dpid(P_i,P_j),\ \ \text{for each distinct $0\leq i,j\leq n-1$ and for each $\pid\notin S$.}\eeq
Then $n\leq (p|S|)^{2D}$.
\end{lemma}

\begin{proof}Let $\pid\notin S$ such that the residue field at $\pid$ has cardinality bounded by $i(p,D,|S|)=(p|S|)^{2D}-1$. We know that such a $\pid$ exists by Lemma \ref{ipDs}. We may assume that for each $i\in\{0,\ldots, n-1\}$, the point $P_i=[x_i:y_i]$ is written in $\pid$--coprime integral coordinates, i.e. $\min\{v_\pid(x_i),v_\pid(y_i)\}=0$ .
Let $x_i^\prime,y_i^\prime\in K$ such that
$$\begin{pmatrix}y_0& -x_0\\y_1& -x_1
\end{pmatrix}\begin{pmatrix}x_i\\y_i
\end{pmatrix}=\begin{pmatrix}x_i^\prime\\y_i^\prime
\end{pmatrix},$$
for each $i\in\{0,1,\ldots,n-1\}$. Let us denote by $P_i^\prime$ the point $[x_i^\prime:y_i^\prime]$. For all distinct $i,j\in\{0,1,\ldots,n-1\}$, we have
$$\dpid([x_i^\prime:y_i^\prime],[x_j^\prime: y_j^\prime])=v_\pid((x_0 y_1-x_0y_1)^{-1}(x_i y_j-x_j y_i))=0$$
Thus $P_0^\prime, \ldots, P_{n-1}^\prime$ are $n$ points whose reductions in $\p_1(k(\pid))$ are pairwise distinct for each $\pid\notin S$.
Then $n\leq |k(\pid)|+1\leq (p|S|)^{2D}$.
 \end{proof}

Suppose that $\phi$ is an endomorphism of $\p_1$ with good reduction outside $S$. Let $P\in \p_1(K)$ be a periodic point for $\phi$. According to Lemma \ref{ipDs} we can take  $\pid\notin S$ such that $|k(\pid)|\leq i(p,D,|S|)$. By Theorem \ref{mst}, there exists a number  $a\leq i(p,D,|S|)^2-1$ such that $P$ is a periodic point for the $a$--th iterate $\phi^a$ with minimal period $p^e$, where $e$ is a non negative integer.  The point $P$ admits $\srs$--coprime coordinates. Then  there is an automorphism $\alpha\in {\rm PGL}_2(K)$, with coefficients in $\srs$  and  inverse having elements in $\srs$, such that $\alpha(P)=[0:1]$ is a periodic point for the map $\alpha\circ \phi\circ \alpha^{-1}$, which has good reduction outside $\hat{S}$ too.
Then we may assume that $P$ is the zero point $[0:1]$ and
the cycle is the following
\beq\label{ncyc} [0:1]\mapsto P_1\mapsto P_2\mapsto \ldots \mapsto P_i\mapsto\ldots \mapsto P_{p^e-1}\mapsto [0:1].\eeq
and we may suppose that $P_i=[x_i:y_i]$ is written in $\srs$--coprime integral coordinates, for each $i\in \{1,\ldots, p^e-1\}$.
As a direct application of Proposition \ref{5.1} and Proposition \ref{6.1}, we have
\beq\label{A_i} \hspace{0.7cm} \dpid(\phi^i(P),P)\geq \min\{\dpid(\phi^i(P),\phi^{i-1}(P)), \ldots, \dpid(\phi(P),P)\}=\dpid(\phi(P),P),\ \hspace{0.1cm} \textrm{ for all } \pid \notin \hat{S}.\eeq
Thus, for each positive integer $i$, there exists $A_i\in \srs$ such that $P_{i}=[A_i x_1:y_i]$.  Furthermore, by Proposition \ref{6.1}, for every $k$ coprime with $p$, we have that $A_k\in \srsu$ and it  can be taken  equal to 1. So $P_k=[x_1:y_k]$ is still written in $\srs$--coprime integral coordinates.

The next lemma is a trivial application of Proposition \ref{6.1} to an iterate  of shape $\phi^{p^k}$.

\begin{lemma}\label{easy}
Let $K$ and $S$ be as above. Suppose that $\phi$ is an endomorphism of $\p_1$ defined over $K$ with good reduction outside $S$. Let $P_0\in\p_1(K)$ be a periodic point of minimal period $p^e$ and $P_i=\phi^i(P_0)$. Then, for any integer of the form $p^k\cdot n$, with $n$ not divisible by $p$ and smaller then $p^{e-k}$, we have
$\dpid(P_0, P_{p^k})=\dpid(P_0, P_{p^k\cdot n})$,
for every ${\mathfrak{p}}\notin S$.
\end{lemma}
%

We are ready to prove our main result about periodic points.

\begin{theorem}\label{cycle}
Let $K$ be a global function field. Let $S$ be non empty finite set of places of $K$ of cardinality $|S|$. Let $p$ be the characteristic of $K$. Let $D=[K:\F_p(t)]$. Then there exists a number $n(p,D,|S|)$, which depends only on $D,p$ and $|S|$,  such that if $P\in\p_1(K)$ is a periodic point for an endomorphism  $\phi$  of $\p_1$ defined over $K$ with good reduction outside $S$, then the minimal period of $P$ is bounded by $n(p,D,|S|)$. We can take
\beq\label{nc}n(p,D,|S|)=\left[(p|S|)^{4D}-1\right]\max\left\{(p|S|)^{2D}, p^{4|S|-2}\right\}.\eeq
\end{theorem}

\begin{proof}
At first let $\phi\colon \p_1\to\p_1$  be as in the statement and with degree $d=1$  (i.e. it is an automorphism).  If a point of $\mathbb{P}_1(K)$ is periodic for $\phi$ with period $n\geq3$, then $\phi^n$ is the identity map of $\mathbb{P}_1(K)$. Hence $\phi$ is given by a matrix in PGL$_2(R_S)$, with two eigenvalues whose quotient is a primitive $n$-th root of unity $\zeta$.
The  degree of the $n$--cyclotomic polynomial $\varphi(n)$ is such that $\varphi(n)^2\geq n-2$, for each positive integer $n$. That inequality follows from some elementary computations involving the  Euler totient function  (e.g. see \cite{A} for definition and properties of this function).
Since $\zeta$ has degree at most $2[K:\F_p(t)]$, then $n\leq  2+4[K:\F_p(t)]^2$ and the last value is smaller than the one in (\ref{nc}).

We now  assume that $\phi\colon \p_1\to\p_1$ is an endomorphism as in the statement, that has degree $d\geq 2$. Let $P\in\p_1(K)$ be a periodic point for $\phi$.  Let $L,\hat{S},\srs,\srsu$ be the ones defined in Subsection \ref{scc}. We denote by $P_i$ the point $P_i=\phi^i(P)$.  As already remarked  (see a few lines before  (\ref{ncyc})), we may assume that the minimal period of $P$ is of the shape $p^e$, up to taking a suitable iterate $\phi^a$ of $\phi$, and that $P=[0:1]$. By applying the divisibility arguments contained in Subsection \ref{dia}, with $L$ and $\hat{S}$ instead of respectively $K$ and $S$, we may assume that the cycle has the  form as in (\ref{ncyc}),  with $P_i= [A_i\cdot x_1: y_i]$ written in $\srs$--coprime integral coordinates, where $A_i,x_1,y_i\in \srs$ and $A_1=1$.
Furthermore, by Proposition \ref{6.1}, for any integer $k$ coprime with $p$, we may take $A_k=1$.

If $v_\pid (A_k)=0$ for each $k\in\{2,\ldots p^e-2\}$ and $\pid\notin \hat{S}$, by Proposition \ref{6.1} and Lemma \ref{=p} we have $p^e\leq (p|S|)^{2D}$.
 Recall that the number $a$, providing the above iterate $\phi^a$, is such that $n=a\cdot p^e$ and $a\leq i(p,D,|S|)^2 -1\leq (p|S|))^{4D}-1$ (this last inequality is not the sharpest one, but it will be useful to get some nice form for our bounds in what follows). Hence $n=a\cdot p^e\leq  \left[ (p|S|)^{4D}-1\right](p|S|)^{2D}$.

Otherwise there exists an index $\alpha$ with $0<\alpha<e$ such that the $A_{p^\alpha}\in\srs\setminus\srsu$. We consider two cases.

\emph{Case $p=2$.}
Assume that $\alpha$ is the smallest integer $k$ such that $A_{2^k}$ is not an $\hat{S}$--unit.
Let $i\equiv 3\mod 4$.
If $\alpha>1$, by Lemma \ref{easy} we have $\dpid(P_1,P_i)=\dpid(P_0,P_1)=\dpid(P_1,P_{2^\alpha})$, for all $\pid\notin \hat{S}$. Then
there exist $u_i, u_{2^\alpha}\in\srsu$ such that $P_i=[x_1:y_1+u_i]$ and $P_{2^\alpha}=[A_{2^\alpha}:  A_{2^\alpha}y_1+u_{2^\alpha}]$. Furthermore, by $\dpid(P_0,P_1)=\dpid(P_i,P_{2^\alpha})$, there exists $u_{i,\alpha}\in\srsu$ such that $A_{2^\alpha}\frac{u_i}{u_{2^\alpha}}-\frac{u_{i,\alpha}}{u_{2^\alpha}} =1$. By Theorem \ref{v} there are at most $r(p,|S|)$ different possible values for $u_i$.
 If $\alpha=1$, we have $\dpid(P_1,P_i)=\dpid(P_0,P_2)$ and $\dpid(P_0,P_1)=\dpid(P_1,P_2)$. Then there exist two $S$--units $u_i, u_{2}$ such that $P_i=[x_1:y_1+A_2u_i]$ and $P_{2}=[A_{2}:  A_{2}y_1+u_{2}]$. As before, we have $\dpid(P_0,P_1)=\dpid(P_i,P_{2})$.  Hence there exists an $u_{i,2}\in \srsu$ such that $A_{2}^2\frac{u_i}{u_{2}}-\frac{u_{i,2}}{u_{2}} =1$. Again, by Theorem \ref{v}, there exist at most $r(p,|S|)$ different possible values for $u_i$.
Note that the positive odd integer $i$ such that $i-1<2^e$ and $4\nmid i-1$ is equal to $2^{e-2}$. Therefore $2^e\leq 4 r(p,|S|)$, i. e. $p^e\leq p^2 r(p,|S|)$ with $p=2$.
Since $r(p,|S|)\leq p^{4|S|-4}$, then it is enough to take
$n(p,D,|S|,d)\leq (i(p,D,|S|)^2-1)(p^2 \cdot r(p,|S|))\leq \left[ (p|S|)^{4D}-1\right]p^{4|S|-2}$.


\emph{Case $p>2$.} Let $b$ be of the shape $b=k\cdot p+i$ with $k\in\{0,1,\ldots p^{e-2}\}$ and $i\in\{2,3,\ldots,p-1\}$.   Because of our assumption on $p$ and Proposition \ref{6.1}, we have $\dpid(P_0,P_{b})=\dpid(P_1,P_{b})=v_\pid(x_1)$ , for any $\pid\notin \hat{S}$. Then there exists an element $u_{b}\in\srsu$ such that
\beq\label{k}P_{b}=[x_1:y_1+u_{b}].\eeq
By $\dpid(P_1,P_{p^\alpha})=v_\pid(x_1)$, we deduce that there exist $u_{p^\alpha}\in \srsu$ such that $P_{p^\alpha}=[A_{p^\alpha}x_1:A_{p^\alpha}y_1+u_{p^\alpha}]$.
Again by Proposition \ref{6.1} we have $\dpid(P_{p^\alpha},P_b)=v_\pid(x_1)$, for every $\pid\notin \hat{S}$. By identity (\ref{k}), there exists $u_{\alpha, b}\in \srsu$ such that $A_{p^\alpha} u_b- u_{p^\alpha}=u_{\alpha,b}$. Observe that there are exactly $(p^{e-2}+1)(p-2)$ of such integers $b$.
 We have that the pair $(u_b/u_{p^\alpha}, u_{\alpha,b}/u_{p^\alpha})\in (\srsu)^2$ is a solution of $A_{p^\alpha}X-Y=1$, where $A_{p^\alpha}\notin \srsu$. By Theorem \ref{v}, there are only $r(p,|S|)$ possible values for $u_b/u_{p^\alpha}$. Hence  $(p^{e-2}+1)(p-2)\leq r(p,|S|)$, i. e. $p^e\leq p^2\left(\frac{r(p,|S|)}{p-2}-1\right)$. Thus $n\leq\left[ (p|S|)^{4D}-1\right]p^{4|S|-2}$, since  $r(p,|S|)\leq p^{4|S|-4}$.
\end{proof}

\section{Bound for the cardinality of a finite orbit}\label{bcfo}
We start by giving some general results that hold for each global field $K$.

The following lemma is a direct application of Proposition \ref{5.1} and Proposition \ref{5.2}.

\begin{lemma}\label{pab}
Let
\begin{equation}\label{cn} P=P_{-m+1}\mapsto P_{-m+2}\mapsto \ldots\mapsto P_{-1}\mapsto P_0=[0:1]\mapsto [0:1]\end{equation}
 be an orbit for an endomorphism $\phi$ defined over $K$, with good reduction outside $S$. For any $a,b$ integers such that $0<a<  b\leq m-1$ and $\pid\notin S$,  we have
\beq\label{eqdis}\dpid(P_{-b},P_{-a})=\dpid(P_{-b},P_0)\leq \dpid(P_{-a},P_0).\eeq

\end{lemma}
\begin{proof}
The inequality in (\ref{eqdis}) follows directly from Proposition \ref{5.2}.
By Proposition \ref{5.1} and  the inequality in (\ref{eqdis}) we have
$\dpid(P_{-b},P_{-a})\geq \min\{\delta_{\mathfrak{p}}(P_{-b},P_{0}),\delta_{\mathfrak{p}}(P_{-a},P_{0})\}=\dpid(P_{-b},P_{0}).$
Let $r$ be the largest positive integer such that $-b+r(b-a)<0$. Then
\\ $\delta_{\mathfrak{p}}(P_{-b},P_{0})\geq \min\{\delta_{\mathfrak{p}}(P_{-b},P_{-a}),\delta_{\mathfrak{p}}(P_{-a},P_{b-2a}),\ldots ,\delta_{\mathfrak{p}}(P_{-b+r(b-a)},P_0)\}=\delta_{\mathfrak{p}}(P_{-b},P_{-a})$.
\end{proof}
We are going to recall some well-known results in the general setting of  non--archimedean dynamics, first the notion of \emph{multiplier}. To ease notation, we use the affine model for endomorphisms of $\p_1(\overline{K})$, that we consider as the set $\overline{K}\cup \{\infty\}$. To any endomorphism $\phi$ of $\p_1(\overline{K})$ we associate the usual rational function defined by $\phi$ on $\overline{K}\cup \{\infty\}$, that, with abuse of notation, we denote with the same symbol. Let $\phi^\prime$ be the usual derivative of $\phi$. We assume that the non-archimedean valuation $v_\pid$ is extended to the whole algebraic closure $\overline{K}$.

\begin{definition}Let $P\in \p_1(\overline{K})$ be a periodic point with minimal periodicity $n$ for the rational function $\phi$. We define $\lambda_P(\phi)$ the multiplier of $P$ as
\begin{displaymath}\lambda_P(\phi) =\begin{cases}(\phi^n)^\prime(P) &\mbox{if}\ P\in\overline{K}\\ \lim_{z\to\infty}\frac{z^2(\phi^n)^\prime(z^{-1})}{(\phi^n(z^{-1}))^2} &\mbox{if}\ P=\infty \end{cases}.\end{displaymath}
We say that $P$ is \emph{attracting} if $v_\pid(\lambda_P(\phi))>0$,  \emph{indifferent} if $v_\pid(\lambda_P(\phi))=0$ and  \emph{repelling} if $v_\pid(\lambda_P(\phi))<0$.
\end{definition}
\noindent The limit in the above defintion exists in $\overline{K}$ (see \cite[Exercise 1.13]{Sil.2}).

 When $\phi$ has good reduction at $v_\pid$, we have the following lemma, that is a trivial application of \cite[Lemma 2.1]{B.2} to a suitable iterate of $\phi$.

\begin{lemma}\label{indatt}
Let $\phi$ be an endomorphism of $\p_1$ define over $K$ with good reduction at $\pid$. Let $P\in\p_1(K)$ be a periodic point. Then $P$ is attracting or indifferent.
\end{lemma}

The next lemma contains some trivial generalizations of \cite[Lemma 2.2]{B.2} and \cite[Lemma 2.3]{B.2}. Roughly speaking, \cite[Lemma 2.2]{B.2} says that if $P\in\p_1(K)$ is an attracting fixed point for a rational function with good reduction at $\pid$, then for any other fixed point $Q\in\p_1(K)$ the reductions of $P$ and $Q$ are distinct in the reduced field $k(\pid)$. If $P$ is an indifferent fixed point, then by \cite[Lemma 2.3]{B.2}, for each preperiodic point $Q\in \p_1(K)\setminus \{P\}$ whose orbit contains $P$, we have that $P$ and $Q$ have distinct reductions $\tilde{P}$ and $\tilde{Q}$ in $\p_1(k(\pid))$.
\par  Let us call \emph{strictly preperiodic} a point that is preperiodic and not periodic.
Furthermore, for any periodic point $P\in \p_
1$, we say that two points $Q_1,Q_2\in \p_1$ are in a same \emph{tail of} $P$ if  $Q_1$ and $Q_2$ belong to a same orbit containing $P$ and are strictly preperiodic.

\begin{lemma}\label{att}
Let $\phi$ be an endomorphism of $\p_1$ defined over $K$ with good reduction at $\pid$.
\begin{itemize}
\item[a)] Let $P\in\p_1(K)$ be an attracting periodic point. Then for every $Q$ different from $P$ in the cycle of $P$ for $\phi$, the reductions $\tilde{P}$ and $\tilde{Q}$ in $\p_1(k(\pid))$ are distinct.
\item[b)] Let $P\in\p_1(K)$ be an indifferent  periodic point. Let  $Q_1,Q_2\in \p_1(K)$ be in the same tail of $P$. Then the reductions $\tilde{Q}_1$ and $\tilde{Q}_2$ in $\p_1(k(\pid))$ are distinct.
\end{itemize}
\end{lemma}
\begin{proof} a) It is sufficient to apply \cite[Lemma 2.2]{B.2} to a suitable iterate of $\phi$.\\
b) As above, for each $Q\in \p_1(K)$ we denote by $\tilde{Q}$ its reduction in $k(\pid)$. Since $\phi$ has good reduction at $\pid$, then  $\phi_\pid(\tilde{Q})=\widetilde{\phi(Q)}$ (e.g. see Theorem 2.18 in \cite{Sil.2} p.59). Suppose $\tilde{Q}_1=\tilde{Q}_2$, then $\tilde{Q}_1$ is a periodic point for the reduced map $\phi_\pid$ and the orbit of $\tilde{Q}_1$ for $\phi_\pid$ coincides with the cycle of $\tilde{P}$. This contradicts \cite[Lemma 2.3]{B.2}, by considering a suitable iterate of $\phi$ instead of $\phi$. \end{proof}


Now we are ready to prove Theorem \ref{preper}.
\begin{proof}[Proof of Theorem \ref{preper}]Let $L$, $\hat{S}$, $\srs$, $\srsu$ the ones defined in Subsection \ref{scc}.

\noindent\emph{Case $K$ function field}  Let $d$ denote the degree of an endomorphism $\phi$ as in the statement of Theorem \ref{preper}. First suppose $d=1$.  So $\phi$ is bijective. Thus, every  preperiodic point is periodic; so it  suffices  to apply Theorem \ref{cycle}.

 Now assume $d\geq 2$. Let $P\in\p_1(K)$ be a preperiodic point for $\phi$.  We take a fixed $\pid_0\notin S$ such that the cardinality $|k(\pid_0)|$ is minimal among the prime not in $S$. By Lemma \ref{ipDs}, we have $|k(\pid_0)|\leq (p|S|)^{2D}-1$.  By Lemma \ref{indatt}, each periodic point for $\phi$ is either indifferent or attracting with respect to the valuation $v_\pid$.  Let $P_0$ be such that  $\phi^m(P)=P_0$ is periodic, where $m$ is the minimum integer such that the point $\phi^m(P)$ is periodic.

If $P_0$ is indifferent with respect $\pid_0$, by Lemma \ref{att}, the reductions modulo $\pid$ of the strictly preperiodic points in the orbit of $P$ are pairwise distinct and all different from the reduction of $P_0$. Therefore we have that
\begin{equation}\label{bind}
|O_\phi(P)|\leq |k(\pid_0)|+n(p,D,|S|)\leq(p|S|)^{4D}\max\left\{(p|S|)^{2D}, p^{4|S|-2},\right\}
\end{equation}
where the number $n(p,D,|S|)$ is the one in Theorem \ref{cycle}.

If $P_0$ is attracting with respect $\pid_0$, then by Lemma \ref{att} we have that the cycle of $P_0$ contains at most $(p|S|)^{2D}$ points. Then, up to taking a suitable conjugate of an $N$--th iterate of $\phi$,  by an automorphism of $\p_1$ associated to a matrix $A\in {\rm SL}(\srs)$, we can assume that the finite orbit of $P$ is the one in (\ref{cn}), with $N\leq  (p|S|)^{2D}$ and each point $P_{-r}=[x_r:y_r]$ is written in $\srs$--coprime coordinates. By Lemma \ref{pab}, for every $1\leq i<j\leq m-1$ there exists $T_{i,j}\in \srs$ such that $x_i=T_{i,j}x_j$. Consider the $\pid$--adic distance between the points $P_{-1}$ and $P_{-j}$. Again by Lemma \ref{pab}, we have
\beq\label{1j}\dpid(P_{-1},P_{-j})=v_\pid(x_1y_j-x_1y_1/T_{1,j})=v_\pid(x_1/T_{1,j}),\eeq
for all $\pid\notin \hat{S}$. Then, there exists $u_j\in\srsu$ such that
\beq\label{yj} y_j=\left(y_1+u_j\right)/{T_{1,j}}.\eeq

 Note that by Lemma \ref{pab} and Lemma \ref{=p}, the number of consecutive points  $P_{-i}$ $(i\geq 0)$  in (\ref{cn}), such that $\delta_\pid(P_0,P_{-i})=\delta_\pid(P_0,P_{-1})$ for each $\pid \notin \hat{S}$, is bounded by the number given in Lemma \ref{=p},  because the conjugation of $\phi$ considered before was made with an automorphism of $\p_1$ associated to a matrix $A\in {\rm SL}(\srs)$.

Suppose that there exists a point $P_{-a}$ of the orbit in (\ref{cn}) such that $v_\pid(x_a)< v_\pid(x_1)$ for a $\pid\notin \hat{S}$.  By the previous argument that involves Lemma \ref{=p}, we can assume that
$a\leq (p|S|)^{2D}-1$.

Consider the $\pid$--adic distance between the points $P_{-a}$ and $P_{-b}$ for any $b>a$.  By Lemma \ref{pab} and (\ref{yj}) with $j=b$, we have:
$\dpid(P_{-a},P_{-b})=v_\pid\left(x_a((y_1+u_b)/T_{1,b})-(x_1/T_{1,b})y_a\right)=v_\pid(x_1/T_{1,b})$,
for all $\pid\notin \hat{S}$. Then there exists $v_b\in \srsu$ such that
\beq\label{fe}\frac{x_1}{x_a{y_1}-x_1{y_a}}v_b-\frac{x_a}{x_a{y_1}-x_1{y_a}}u_b=1.\eeq
For our assumption on $x_a$, we have that the above equation (\ref{fe}) is not $S$--trivial. Therefore, by Theorem \ref{v}, there are only $r(p,|S|)$ possible values for $u_b$. So we have that the number $m$ of points as in (\ref{cn}) verifies
$m\leq a+1+r(p,|S|)\leq (p|S|)^{2 D}+r(p,|S|)$.
As before we can take $r(p,|S|)\leq p^{4|S|-4}$,
then
$$|O_\phi(P)|= N\cdot m\leq \left(p|S|\right)^{2 D}\left( \left(p|S|\right)^{2D}+p^{4|S|-4}\right)$$
and so it is bounded by the number in (\ref{eta}).

\noindent\emph{Case $K$ number field}.
The proof of Theorem \ref{preper} with $K$ a number field is almost the same as the one in the case of function fields.
 We take a fixed $\pid_0\notin S$ such that the cardinality $|k(\pid_0)|$ is minimal among the prime not in $S$. Let $p_0$ be the charachteristic of $k(\pid_0)$. By taking the bound in \cite[Theorem 4.7]{A} we have that
$p_0< 12\left(|S|\log |S|+|S|\log(12/e)\right)< 12 |S|\log (5|S|)$,
because $S$ contains at most $|S|-1$ non archimedean valuations. Then
\begin{equation}\label{bkp}|k(\pid_0)|+1\leq \left(12 |S|\log (5|S|)\right)^D. \end{equation}
As in the case of function field, we first assume that $P_0$ is an indifferent periodic point with respect to $\pid_0$. By applying Theorem \ref{msb} and Lemma \ref{att}, we have
\begin{center}$|O_\phi(P)|\leq |k(\pid_0)|+\left(12(|S|+1)\log(5(|S|+1))\right)^{4D}\leq \left(12(|S|+2)\log(5(|S|+1))\right)^{4D}.$\end{center}
Assume that $P_0$ is attractive, with respect to the prime $\pid_0$. The proof uses the same arguments as in the case of function fields.
%
%
%
%
But here it is enough to take $a=2$ and apply Theorem \ref{v} with $\Gamma=\left(\sqrt{R_S^*}\right)^2$ for the equation (\ref{fe}).  Since $\Gamma$ has rank $2|S|-2$, then the units $u_j$ assume at most $2^{16|S|-8}$ values. Therefore
$m\leq 2^{16|S|-8}+3$
 and by applying Theorem \ref{msb} we obtain  $|O_\phi(P)|\leq\left( 2^{16|S|-8}+3\right) \left(12 |S|\log (5|S|)\right)^D$.
\end{proof}

\subsection{Proof of Corollary \ref{UBCforCGR}}\label{bcspp}
Let $C$ be an upperbound for the minimal periodicity of a point in $\p_1(K)$ for an endomorphism $\phi$ defined over $K$ of degree $d\geq 2$. More generally, let $B$ be an upper bound for the cardinality of a finite orbit in $\p_1(K)$ for $\phi$. One can prove a bound $b(B,C,d)$, that depends only on $B,C$ and $d$, for the cardinality of the set ${\rm PrePer}(\phi,K)$. For example, take $\mathcal{P}$ the set of all primes in $\mathbb{Z}$ and
\begin{center}$n=\prod_{p\in\mathcal{P}} p^{m_p(C)}$\end{center}
where $m_p(C)=\max \{{\rm ord}_p(z)\mid z\in\N,\ z\leq C\}$. Each periodic point is either a solution of $\phi^n(P)-P=0$ or the point at infinity. Then $\phi$ has at most $d^n+1$ finite orbits in $\p_1(K)$ (this is a very rude upperbound). Thus a bound for the cardinality of ${\rm PrePer}(\phi,K)$ is $d^B\cdot(d^n+1)$.

\subsection{Proof of Corollary \ref{n3}}
Since $S$ contains only the archimedean place, then $R_S^*=\{1,-1\}$.
Let $P$ be a periodic point for $\phi$ of minimal period $n$. As usual we may assume that $P=[0:1]$. Let $p$ be a prime dividing $n$. Then $n=p^em$, for some positive integers $e$ and $m$, where
$m$ is coprime with $p$. Thus, the iterate $\phi^{m}$  has the following  cycle of length $p^e$
\begin{equation} \label{c1} [0:1]\mapsto P_1 \mapsto ... \mapsto P_{p^{e}-1}\mapsto [0:1]. \end{equation}
We may assume that the point $P_i=[x_i:y_i]$  is written in integral coprime coordinates for each index $i$.
 By Proposition \ref{6.1}, for each $2\leq i \leq p-1$, there exists $u_i\in R_S^*$ such that
$[x_i:y_i]=[x_1:y_1+u_i]$. If $p\notin \{2,3\}$, then the beginning of the cycle \eqref{c1} is
$$  [0:1] \mapsto [x_1:y_1] \mapsto [x_1:y_1+u_2]   \mapsto [x_1:y_1+u_3]  \mapsto [x_1:y_1+u_4] \mapsto ... $$
\noindent for some $u_2,u_3,u_4\in R_S^*$. Since $R_S^*=\{1,-1\}$, we have $|\{P_2,P_3,P_4\}|\leq 2$. Then $n=2^{\alpha}3^{\beta}$ for some integers $\alpha$ and $\beta$.
Up to taking a suitable iterate of the map $\phi$, we may treat separately the cases when $n=2^{\alpha}$ and  when $n=3^{\beta}$.
Assume that $n=2^{\alpha}$. We are going to prove that $\alpha\leq 1$. Suppose that $\alpha\geq 2$.
By considering the $p$--adic distances $\delta_p(P_1,P_i)$ with $2\leq i\leq 4$, by Proposition \ref{6.1}, we get that the beginning of the cycle is
$$[0:1]\mapsto [x_1:y_1]\mapsto [A_1x_1:A_1y_1+u_2]\mapsto [x_1:y_1+A_1u_3]\mapsto \ldots$$where $A_1\in R_S$, $u_2,u_3\in R_S^*$ and everything is written in coprime integral coordinates.

Again by Proposition \ref{6.1} we have $\delta_p(P_2,P_3)=\delta_p(P_0,P_1)$ for every  prime $p$; then there exists an $S$--unit $u_{2,3}$ such that $A_1^2u_3=u_2+u_{2,3}$. Since $R_S^*=\{1,-1\}$, we have $A_1^2\in\{0,2,-2\}$, then $A_1=0$,
that contradicts $\alpha\geq 2$. Then $\alpha\leq 1$.

Assume that $n=3^{\beta}$. We are going to prove that $\beta\leq 1$. Assume that $\beta\geq 2$. As before, by the divisibility properties
listed in Proposition \ref{6.1} and by considering the $p$--adic distances $\delta_p(P_1,P_i)$ with $2\leq i\leq 4$, we have that the beginning of the cycle is
$$  [0:1]\mapsto [x_1:y_1]\mapsto [x_1:y_1+u_2]\mapsto [A_1x_1:A_1y_1+u_3]\mapsto [x_1:y_1+A_1u_4]\mapsto\ldots
$$
where $A_1\in R_S$ and $u_2,u_3,u_4\in R_S^*$ and everything is written in coprime integral coordinates.
By the second part of Proposition \ref{6.1}, we have $\delta_p(P_4,P_3)=\delta_p(P_0,P_1)$ for every  prime $p$.
Then there exists an $S$--unit $u_{3,4}$ such that
$A_1^2u_4=u_3+u_{3,4}$. Since $R_S^*=\{1,-1\}$, we have $A_1^2\in\{0,2,-2\}$, that contradicts $\beta\geq 2$; so $\beta\leq 1$.

Thus  we have proved that $n\in \{1,2,3,6\}$.
If $n=6$, with few calculations we see that the cycle has the form
\begin{center}$[0:1]\mapsto [x_1:y_1]\mapsto [A_2x_1:y_2]\mapsto [A_3x_1:y_3]\mapsto [A_2x_1:y_4]\mapsto
[x_1:y_5]\mapsto [0:1]$\end{center}
where $A_2,A_3\in R_S$ and everything is written in coprime integral coordinates.
We may apply Proposition \ref{6.1}. Then, by considering the $p$--adic distances $\delta_p(P_1,P_i)$ for all indexes $2\leq i\leq 5$ for every prime $p$, we get that there exists  an $S$--units $u_i$ such that
\begin{equation} \label{beginning4}y_2=A_2y_1+u_2;\ \ y_3=A_3y_1+A_2u_3;\ \ y_4=A_2y_1+A_3u_4;\ \ y_5=y_1+A_2u_5.\end{equation}
%

Furthermore  $\delta_p(P_2,P_4)= \delta_p(P_0,P_2)$ for every  prime $p$. Thus there exists $u_{2,4}\in R_S^*$ such that  $y_4=y_2+u_{2,4}$ and by (\ref{beginning4}) one sees that
 $A_2y_1+A_3u_4=A_2y_1+u_2+u_{2,4}$. Since $R_S^*=\{1,-1\}$, then the equality $A_3u_4=u_2+u_{2,4}$ implies $A_3\in \{2,-2\}$.
 By $\delta_p(P_2,P_5)=\delta_p(P_0,P_3)$, we have that there exists  $u_{2,5}\in R_S^*$ such that $A_2y_5=y_2+A_3u_{2,5}$. By substituting in the last equality  the
 expressions of $y_2$ and $y_5$ appearing in \eqref{beginning4}, we have $A_2^2u_5=u_2+A_3u_{2,5}$. Since $A_2^2$ is
 a square and $A_3\in\{-2,2\}$, then the  only possibility is $A_2^2=1$. Without loss of generality we may assume $A_2=1$. In particular, we have $y_3=A_3y_1+u_3$ and
  $y_4=y_1+A_3u_4$. By considering $\delta_p(P_3,P_4)=\delta_p(P_0,P_1)$,
 we obtain that there exists $u_{3,4}\in R_S^*$ such that $A_3^2u_4=u_3+u_{3,4}$. As above we have a contradiction with $n=6$, so we  conclude that $n\leq 3$.

Suppose now that $P\in\p_1(\Q)$ is a preperiodic point for $\phi$.    Let $k$ be the cardinality of the cycle in the orbit of $P$, so  $k\leq 3$. At first we assume that the orbit of $P$ contains an indifferent periodic point $P_0$ with respect the prime $2$.  By Lemma \ref{att} we have that the strictly  preperiodic points in the orbit of $P$ are at most $2$.  Hence  $|O_{\phi}(P)|\leq 5$.  Assume now that  the orbit of $P$ contains  only  attractive periodic points. Let $m$ be the minimum integer such that $\phi^{m\cdot k}(P)$ is a fixed point for $\phi^k$ and denote it with $P_0$. Without loss of generality, we can assume $P_0=[0:1]$. Furthermore, we may assume that the orbit is as in \eqref{cn}, with $P_{-j}=[x_{j}:y_{j}]$, written in coprime integral coordinates, for all integer $j$ with $m>j>0$.  By using the same arguments as in the proof of Theorem \ref{preper}, we see that also  in this case the identities as in (\ref{yj}) hold.  Then, for each $j$  with $m>j>0$, there exists an integer $T_{1,j}$ such that $x_1=T_{1,j}x_j$ and (\ref{yj}) holds. We have
$P_{-j}=[x_j:y_j]=[x_1:y_1+u_j]$, for a suitable $S$--unit $u_j$. Since $|R_S^*|=2$,  we conclude $m\leq 4$. Thus, there are at most $4$ preperiodic points in the orbit of $P$ for $\phi^k$.
 Since $k\leq3$, we have $|O_{\phi}(P)|\leq 3\cdot 4=12$.

\end{document}